\documentclass[11pt,draft,reqno]{amsart}

\usepackage{amsthm,array,amssymb,amscd,amsfonts,latexsym}

\DeclareFontFamily{OT1}{wncyr}{\hyphenchar\font45
}
\DeclareFontShape{OT1}{wncyr}{m}{n}{%
   <5> <6> <7> <8> <9> gen * wncyr
   <10> <10.95> <12> <14.4> <17.28> <20.74>  <24.88>wncyr10}{}
\DeclareFontShape{OT1}{wncyr}{m}{it}{%
   <5> <6> <7> <8> <9> gen * wncyi
   <10> <10.95> <12> <14.4> <17.28> <20.74> <24.88> wncyi10}{}
\DeclareFontShape{OT1}{wncyr}{m}{sc}{%
   <5> <6> <7> <8> <9> <10> <10.95> <12> <14.4>
   <17.28> <20.74> <24.88>wncysc10}{}
\DeclareFontShape{OT1}{wncyr}{b}{n}{%
   <5> <6> <7> <8> <9> gen * wncyb
   <10> <10.95> <12> <14.4> <17.28> <20.74> <24.88>wncyb10}{}
\input cyracc.def
\def\rus{\usefont{OT1}{wncyr}{m}{n}\cyracc\fontsize{8}{10pt}\selectfont}

\DeclareMathSizes{9}{9}{7}{5} 

\theoremstyle{plain}

\newtheorem{theorem}{Theorem}[section]

\newtheorem{lemma}[theorem]{Lemma}

\theoremstyle{remark}
\newtheorem{remark}[theorem]{Remark}

\theoremstyle{definition}

\newtheorem{example}[theorem]{\bf Example}

\newtheorem{nothing*}[theorem]{}

\begin{document}

\title[Two orbits: When is one
in the closure of the other? ]{Two orbits:\\
When is one in the closure of the other? }

\author[Vladimir L. Popov]{Vladimir L. Popov}
\address{Steklov Mathematical Institute,
Russian Academy of Sciences, Gubkina 8, Moscow\\
119991, Russia} \email{popovvl@orc.ru}

\thanks{Supported by Russian grants {\rus RFFI
08--01--00095}, {\rus N{SH}--1987.2008.1}, and a
granting program {\it Contemporary Problems of
Theoretical Mathematics} of the Mathe\-matics
Branch of the Russian Academy of Sciences.}

\dedicatory{To V. A. Iskovskikh on the occasion
of his $70$th birthday}

\begin{abstract}
Let $G$ be a connected linear algebraic group,
let $V$ be a finite dimensional algebraic
$G$-module, and let $\mathcal O_1$, $\mathcal
O_2$ be two $G$-orbits in $V$. We describe a
constructive way to find out whether  or not
$\mathcal O_1$ lies in the closure of~$\mathcal
O_2$.
\end{abstract}

\maketitle

\section{\bf Introduction}

\begin{nothing*}\label{begin}
Fix  an algebraically closed ground field $k$ of
arbitrary characteristic.

Let $G$ be a connected linear algebraic group
and let $V$ be a finite dimensional al\-gebraic
$G$-module. Consider two points $a$ and $b\in V$
and their $G$-orbits $G\cdot a$ and~$G\cdot b$.

The following problem continually arises in
algebraic transformation group theory and its
applications:
\begin{gather}
\hskip -18mm\begin{align*}
&\mbox{\it How can one find out whether or not
the orbit
$G\cdot a$ lies in }\\[-3pt] &\mbox{\it
the closure in $V$ of
 the orbit
 $G\cdot b$}?
\end{align*}\tag{$\ast$}
\end{gather}
(here and further topological terms
are related to the Zarisky topology).

\begin{example}
If the group $G$ is reductive and $a=0$, then
Problem $(\ast)$ means finding out whether or
not  the point $b$ is unstable in the sense of
Geometric Invariant Theory
 \cite{M}. A description of the cone of unstable
 points is provided by the Hilbert--Mumford theory.
\end{example}

\begin{example} Let $G$ be a torus and let
$X(G)$ be the group of its characters in
additive notation. For every $\lambda\in X(G)$,
$g\in G$, and $v\in V$ denote by $g^\lambda$ and
$v_\lambda$ respectively the value of $\lambda$
at $g$ and the projection of $v$ to the
$\lambda$-wight subspace of the $G$-module $V$
parallel to the sum of the other weight
subspaces. Let ${\rm supp}\,v:=\{\lambda\in
X(G)\mid v_\lambda\neq 0\}$. Then by \cite{PV72}
Problem $(*)$ means finding out whether or not
the following conditions hold: (i) the cone
generated by ${\rm supp}\,a$ in
$X(G)\otimes_\mathbb Z\mathbb R$ is a face of
the cone generated by ${\rm supp}\,b$, and (ii)
there is an element $g\in G$ such that
$g^\lambda a_\lambda=b_\lambda$ for every
$\lambda\in {\rm supp}\,a$.
\end{example}

\begin{example} If the group $G$ is unipotent,
then every $G$-orbit is closed in $V$, see
\cite{R61}. Therefore Problem $(*)$ means
finding out whether or not  the
points $a$ and $b$
lie in
one and the same $G$-orbit.
\end{example}

\begin{example}\label{15}
Let ${\rm char}\,k=0$. Assume that $G$ is a
simple group, $V$ is its Lie algebra endowed
with the adjoint action of $G$, and the elements
$a$ and $b$ are nilpotent. If $G$ is a classical
group (i.e., of type ${\sf A}_l$, ${\sf B}_l$,
${\sf C}_l$, or ${\sf D}_l$), then the answer to
$(\ast)$ is given by the known rule formulated
in terms of the sizes of Jordan blocks of the
Jordan normal forms of
 $a$ and $b$, see, e.g.,\;\cite{CM}. If the
 group $G$ is exceptional (i.e., of type ${\sf E}_6$, ${\sf E}_7$, ${\sf
E}_8$, ${\sf F}_4$, or ${\sf G}_2$), then this
answer, obtained by means of ad hoc methods, is
given by the explicit Hasse diagrams of the set
of nilpotent orbits endowed with the Bruhat
order, i.e., partially ordered according to the
rule
\begin{equation*} \label{ne}
{\mathcal O}_1\leqslant {\mathcal O}_2\iff
{\mathcal O}_1\subseteq \overline{{\mathcal
O}_2}
\end{equation*}
(as usual, bar means the closure in $V$), see
\cite{S82}, \cite{C}.

\begin{example} Apart from the classical case of
orbits of a Borel subgroup of a reductive group $G$ on
the generalized flag variety $G/P$, and also the case
of Example \ref{15}, the Hasse diagrams of the sets of
orbits endowed with the Bruhat order are found
utilizing the ad hoc methods in some other special
cases, see, e.g., \cite{K90}, \cite{Pe04},
\cite{BHRZ99}, \cite{GHR07}, \cite{MWZ99}, and
Examples \ref{LL} and \ref{M} below. On the other
hand, in a number of cases the orbits are classified,
but the Hasse diagrams are not found: for instance,
this is so for nilpotent $3$-vectors of the
$n$-dimensional spaces where $n\leqslant 9$, for
$4$-vectors of an $8$-dimensional space, for spinors
of the $m$-dimensional spaces where $m\leqslant 14$
and $16$, see the relevant references in \cite{PV94}.
\end{example}
\end{example}
\begin{example}\label{LL}
Let $L$ be a finite dimensional vector space over $k$.
Let  $G={\rm GL}(L)$ and $V=L^*\otimes L^*\otimes L$.
Points of $V$ are the structures of  (not necessarily
associative) $k$-algebras on the vector space $L$. The
algebras defined by the struc\-tu\-res
 $a$ and $b$ are isomorphic if and only if
$G\cdot a=G\cdot b$. In the language of the
theory of algebras Problem $(\ast)$ is
formulated as follows: How can one find out
whether or not  the algebra defined by the
structure $a$ is a {\it degeneration} of the
algebra defined by the structure $b$? In general
case it is considered to be a difficult problem.
There is a number of papers where a
classification of degenerations in various
special cases is obtained by means of  ad hoc
methods, see, e.g., \cite{B05}, \cite{BS},
\cite{S90}, and survey \cite[Chap.\,7]{OVG}.
\end{example}

\begin{example}\label{M} Consider the conjugation action
of the group $G={\rm GL}_d(k)$ on
 ${\rm Mat}_{d, d}(k)$. Consider
a finite dimensional associative $k$-algebra $A$ and a
$d$-dimensional vector space $L$ over $k$. If a basis
in $A$ and a basis in $L$ are fixed, then the set of
structures of left $A$-modules on $L$ is naturally
identified with a closed invariant subset ${\rm
Mod}_A^d$ of the direct sum of $\dim_k\!A$ copies of
the $G$-module ${\rm Mat}_{d, d}(k)$. Denote by $M_x$
the $A$-module corresponding to a point $x\in {\rm
Mod}_A^d$. Then the $A$-modules $M_a$ and $M_b$ are
isomorphic if and only if
 $G\cdot a=G\cdot b$, and in
this theory
 the condition $G\cdot
a\subseteq\overline{G\cdot b}$ is expressed  by
saying that  $M_a$ is a {\it degeneration} of
$M_b$. If $A$ is the path algebra of a quiver
obtained by fixing an orientation of the
extended Dynkin graph of a root system of type
 ${\sf A}_l$, ${\sf D}_l$, ${\sf E}_6$, ${\sf
E}_7$, or ${\sf E}_8$, then in \cite{B95} a
characterization of the degeneration relation in
terms of $A$-module structures of
 $M_a$ and $M_b$ and an algorithm that
 finds out whether or not
$M_a$ is the degeneration of $M_b$ are obtained.
\end{example}

\begin{example} \label{complexity} In case
of the natural action of the group ${\rm
GL}_n(k)$ on the space of $n$-ary forms of
degree $d$
with the coefficients in $k$ Problem $(\ast)$
(under the name {\it The orbit closure problem})
is of fundamental importance in application of
geometric invariant theory
 to
complexity theory,
see \cite{MS01}.
\end{example}
\end{nothing*}

\begin{nothing*} In this paper we give a
constructive solution to Problem $(*)$: in Section 2
we suggest an algorithm that provides an answer to
$(*)$ by means of a finite number of effectively
feasible operations. Namely, we explicitly point out a
finite system of linear equations in finitely many
variables over the field $k$ such that the inclusion
 $G\cdot a\subseteq \overline{G\cdot b}$
 is equivalent to its inconsistency
  (the precise formulation is contained in Theorem
   \ref{thm}). By Kronecker--Capelli theorem
this reduces answering $(*)$ to comparing ranks of two
explicitly given matrices with the coefficients in $k$
that can be executed constructively. Of course, the
existence of a constructive solution to Problem $(*)$
immediately leads to the problem of finding a most
effective algorithm. But this is another problem that
we do not consider here.
\end{nothing*}
\begin{nothing*}\label{L} In Section 3
we suggest another algorithm for finding an
answer to Problem $(\ast)$. It is less effective
than the algorithm from Section 2, but it
provides more information and concerns a more
general problem. To wit, let
 $L$ be a linear subvariety of $V$.
 We show how one can constructively find
 a finite
 system of polynomial functions
 $q_1,\ldots, q_m$ on
$V$ such that
\begin{equation}\label{qqq}
\overline{G \cdot  L}= \{x\in V\mid
q_1(x)=\ldots= q_m(x)=0\}.
\end{equation}
For $L=b$ this provides the following
constructive answer to Problem $(\ast)$:
\begin{equation*}
G \cdot  a\subseteq \overline{G \cdot  b} \iff
q_1(a)=\ldots=q_m(a)=0.
\end{equation*}

Note that varieties of the form $\overline{G\cdot
L}$ are ubiquitous in algebraic transforma\-tion
group theory: apart from orbit closures, to them
also belong irreducible components of Hilbert
null-cones and, more generally, closures of
Hesselink strata \cite{Po03}, closures of sheets
\cite{PV94}, and closures of Jordan  (a.k.a.
decom\-po\-si\-tion) classes \cite{TY05}. Also
note that if a system of polynomials
$q_1,\ldots\break\ldots, q_m$ satisfying
\eqref{qqq} is given, modern commutative algebra
provides algorithms to constructively find a
system of generators of the ideal of all
polynomials vanishing on $\overline{G \cdot L}$,
see, e.g.,\,\cite[Chap.\,4, \S2]{CLO98}. In
particular, this provides methods to
construc\-ti\-vely find generators of the ideal
of polynomials vanishing on the closure of orbit.
In some special cases (for instance, for
nilpotent orbits of the adjoint action of the
group
 ${\rm SL}_n(k)$
and for ``rank varieties'') such generators have
been found, see \cite{W89}.
\end{nothing*}

\begin{nothing*} In essence, both algorithms are based
on the possibility to rationally para\-met\-rize
an open subset of $G$ by means of a variety of
the form
\begin{equation*}
\mathbb A^{r, s}:=\{(\varepsilon_1,\ldots,
\varepsilon_{r+s})\in \mathbb A^{r+s}\mid
\varepsilon_1\cdots\varepsilon_r\neq 0\},\qquad
r, s\in \mathbb N
\end{equation*}
(we denote by $\mathbb N$ the set of all
nonnegative integers), more precisely, on the
existence of a dominant morphism
\begin{equation}\label{iota}
\iota\colon \mathbb A^{r, s} \rightarrow G.
\end{equation}
\end{nothing*}
\begin{nothing*} As every normal quasiprojective
variety endowed with an algebraic ac\-ti\-on of
$G$ can be  equivariantly embedded in a
projective space  \cite{PV94}, in algebraic
transformation group theory and its applications
continually arises the problem analogous to
$(\ast)$, but for an action of $G$ on a
projective space (in fact, this is so in Example
\ref{complexity}). However, this problem is
reduced to Problem $(\ast)$ for actions on
vector spaces, see Subsection \ref{reduction}.
\end{nothing*}

\begin{nothing*} We close this introduction
by noting that as
$$
G\cdot a =G\cdot b\iff G\cdot a\subseteq
\overline{G\cdot b}\hskip 2mm \mbox{and}\hskip 2mm
G\cdot b\subseteq \overline{G\cdot a},
$$
a constructive solution to Problem $(*)$ provides a
constructive solution to the following problem:
\begin{align*}&\mbox{\it How can
one find out whether or not  two given points
of~$\;V$ lie }\\[-3pt] &\mbox{\it  in one  and
the same $G$-orbit}? \end{align*}

This means that  our result yields a constructive
solution to the classification problem for some types
of mathematical objects: for instance, for
$k$-algebras of a fixed dimension up to isomorphisms
(see Example \ref{LL}); for $A$-modules of a fixed
dimension over a fixed $k$-algebra $A$ up to
isomorphism (see Example \ref{M}); for
$k$-representations of a fixed dimension of a given
quiver; for some types of algebraic varieties (see
Example \ref{A}).
\end{nothing*}

\begin{example}\label{A} Let $f_1$ and $f_2$ be two irreducible
forms of the same degree in the homogeneous
coordinates of the projective space $\mathbb P^n$.
Assume that, for every $i=1, 2$, the hypersurface
$H_i$ in $\mathbb P^n$ defined by the equation $f_i=0$
is smooth. Let $n\geqslant 4$. Then by a theorem of
Severy--Lefschetz--Andreotti every positive divisor on
the hypersurface $H_i$ is cut out by a hypersurface in
$\mathbb P^n$, see \cite[Theorem 2]{MM64}. It is not
difficult to deduce from this that the algebraic
varieties $H_1$ and $H_2$ are isomorphic if and only
if $H_1$ is the image of $H_2$ under a projective
transformation of $\mathbb P^n$, i.e., if and only if
$f_1$ is in the ${\rm GL}_{n+1}$-orbit of $f_2$.
\end{example}
\begin{remark}\label{AA} There are the
other types of algebraic varieties for which the
isomorphism problem is reduced to finding out whether
or not some forms lie in one and the same orbit of the
corresponding linear algebraic group. For instance,
smooth projective curves of a genus $g\geqslant 2$ are
embedded into $\mathbb P^{5g-6}$ be means of the
tripled canonical class, and two curves are isomorphic
if and only if the image of one of them is transformed
to the image of the other by a projective
transformation of $\mathbb P^{5g-6}$. In turn, the
latter condition is equivalent to the property that
the Chow forms (a.k.a. Cayley forms) of these images
lie in one and the same orbit of the corresponding
linear algebraic group.
\end{remark}

\section{\bf Main result}

\begin{nothing*} Our further considerations
are based on the following fact.
\begin{lemma}\label{domi}  For some
$r, s\in\mathbb N$, there is a dominant morphism
\eqref{iota}. Moreover, for $r={\rm rk}\, G$,
there is  an open embedding \eqref{iota}.
\end{lemma}
\begin{proof}
Let $R_u(G)$ be the unipotent radical of the
group  $G$. By \cite{R56}, as $R_u(G)$ is a
connected solvable group, the canonical
projection $G\to G/R_u(G)$ is a torsor with the
base $G/R_u(G)$ and the structural group
$R_u(G)$ that is locally trivial in the Zariski
topology. By \cite{G58}, as the base is affine
and the structural group is connected and
unipotent, this torsor is trivial. Hence the
variety $G$ is isomorphic to the product of
varieties $R_u(G)$ and $G/R_u(G)$. As $R_u(G)$
is a connected unipotent group, the first of
them is isomorphic to  $\mathbb
A^{\dim\,R_u(G)}$, see \cite{G58}. On the other
hand, the big Bruhat cell of the reductive group
$G/R_u(G)$ is isomophic to $\mathbb A^{r, \dim
G/R_u(G)-r}$, where $r={\rm rk}\, G/R_u(G)= {\rm
rk}\,G$, see \cite{Spr98}. As the variety
$\mathbb A^{r, \dim G-r}$ is isomorphic to
$\mathbb A^{\dim\,R_u(G)}\times \mathbb A^{r,
\dim G/R_u(G)-r}$, this shows that there is its
open embedding in~$G$.
\end{proof}
\end{nothing*}

\begin{nothing*} \label{o}{\it Notation}

\vskip 1mm

Fix a basis $e_1,\ldots, e_{n}$ in $V$. As the
case $n=1$ is clear, in the further we assume
that $n>1$. There are functions $\rho_{i, j}$,
$1\leqslant i, j\leqslant n$, regular on
 $G$ such that the action of $G$
on $V$ is given by the matrix representation
\begin{equation}\label{matr}
\rho\colon G\to {\rm Mat}_{n, n}(k),\hskip
3mm\rho(g)=
\begin{bmatrix}\rho_{1,1}(g)&\cdots&\rho_{1,n}(g)\\
\hdotsfor[2]{3}\\
\rho_{n,1}(g)&\cdots&\rho_{n,n}(g)
\end{bmatrix},\hskip 2mm g\in G,
\end{equation}
i.e., $\rho (g)$ is the matrix of the linear
$V\to V$, $v\mapsto g\cdot v$ in the basis
$e_1,\ldots,e_{n}$, so that
\begin{equation}\label{actionn} g\hskip -.3mm\cdot\hskip
-.3mm \Bigl(\sum_{i=1}^{n}
\gamma_ie_i\Bigr)=\sum_{i=1}^{n}\Bigl(\sum_{j=1}^{n}\rho_{i,
j}(g)\gamma_j\Bigr)e_i\hskip 3mm\mbox{для любых
$g\in G$ и $\gamma_1,\ldots, \gamma_{n}\in k$}.
\end{equation}

 Fix a dominant morphism
\eqref{iota}: this is possible by Lemma
\ref{domi}. Denote by  $x_1,\ldots, x_{r+s}$ the
standard coordinate functions on
 $\mathbb A^{r, s}$:
\begin{gather}\label{coord}
x_i(a)=\varepsilon_i\hskip 2mm \mbox{for
$a=(\varepsilon_1,\ldots,\varepsilon_{r+s})\in
\mathbb A^{r, s}$.}
\end{gather}
As $x_1,\ldots, x_{r+s},
x_1^{-1},\ldots,x_{r}^{-1}$
generate the $k$-algebra $k[\mathbb A^{r, s}]$
of regular functions on $\mathbb A^{r, s}$ and
$x_1,\ldots, x_{r+s}$ are algebraically
independent over $k$, all the monomials of the
form
\begin{equation}\label{monom}
x^{i_1}_{1}\cdots x^{i_{r+s}}_{r+s}, \hskip
3mm\mbox{где \hskip 2mm $i_1,\ldots, i_{r+s}\in
\mathbb Z \hskip 1.5mm \mbox{и}\hskip 1.8mm
i_{r+1},\ldots, i_{r+s}\in \mathbb N$,}
\end{equation}
constitute a basis of the vector space
$k[\mathbb A^{r, s}]$ over $k$.
\end{nothing*}

\begin{nothing*} {\it The degree of
the variety $\rho(G)$}

\vskip 1mm

Recall \cite{M76} that the degree of a locally
closed subset  $Y$ of $\mathbb A^{l}$ is the
cardinality  ${\rm deg}\,Y$ of the intersection
of $Y$ with an $(l-\dim\,Y)$-dimensional linear
subvariety of $\mathbb A^l$ in general position.
For us, the degree ${\rm deg}\,\rho(G)$ of the
subvariety $\rho(G)$ of the space of matrices
${\rm Mat}_{n, n}(k)$ is of the special
interest. In the important case where ${\rm
char}\,k=0$ and $G$ is a reductive group there
is the following formula for computation of this
number.

Fix a maximal torus $T$ in $G$. Let $X(T)$ be
its character group in additive notation. The
latter is a free abelian group of rank $r=\dim
T={\rm rk}\,G$ naturally embedded in the real
vector space $E=X(T)\otimes_{\mathbb Z}\mathbb
R$. Fixing a basis in $X(T)$, we fix an
isomorphism between $E$ and the coordinate space
$\mathbb R^r$. Identify these spaces by means of
this isomorphism. Then the group $X(T)$ is
identified with the lattice $\mathbb Z^r$ in
$\mathbb R^r$. The Weyl group $W:=N_G(T)/T$
naturally acts on $E$. We denote by $d\nu$ the
standard volume form on $E$. Let $\mathcal P_V$
be the convex hull in $E$ of the union of zero
and the system of weights of the $T$-module $V$.
Replacing $G$ by the quotient group of $G$ by
the unity connected component of the kernel of
the $G$-action on $V$, we may (and shall) assume
that this kernel is finite and hence $\dim
\mathcal P_V=r$. Fix a system $R_+$ of positive
roots of the root system of $G$ with respect to
$T$. For every root $\alpha\in R_+$, denote by
$\alpha\!^{\vee}$ the corresponding coroot,
i.e., the linear form on $E$ defined by the
formula $\alpha\!^{\vee}\colon E\to \mathbb R,
\hskip 3mm \alpha\!^{\vee}\,(v)=2\langle
\alpha\,\vert\, v\rangle/\langle \alpha\,\vert\,
\alpha\rangle, $ where $\langle\hskip 1.8mm
\mid\hskip 1.8mm \rangle$ is a $W$-invariant
inner product on $E$, see \cite{B68}. Let
$m_1+1, \ldots, m_r+1$ be the set of degrees of
homogeneous free generators of the algebra of
$W$-invariant polynomial functions on the
space
$E$, i.e., $m_1,\ldots, m_r$ are the
exponents of  $W$, see \cite{B68}.

\begin{theorem}[{{\rm Kazarnovski\v\i~\cite{K87}}}]
Let ${\rm char}\,k=0$ and let $\rho$ be a
represen\-ta\-tion of a connected reductive
group
 $G$ with finite kernel. Then
\begin{equation}\label{degree}
{\rm deg}\,\rho(G):=\frac{\dim G\hskip
.25mm!}{|W|(m_1!\cdots m_r!)^2|\!\ker
\rho|}\int_{\mathcal P_V}\prod_{\alpha\in
R_+}(\alpha\!^{\vee}\,)^2d\nu.
\end{equation}
\end{theorem}

\begin{example}\label{sl2} Let ${\rm char}\,k=0$.
Consider the main object of pre-Hilbertian classical
invariant theory: $G={\rm SL}_2(k)$ and $V=V_h$ is the
space of binary forms of degree  $h$ in variables
$z_1, z_2$ over the field $k$ on which $G$ acts by
linear substitutions of variables:
\begin{equation}\label{sl}
g\cdot z_1=\alpha z_1+\gamma z_2,\hskip 2mm
g\cdot z_2=\beta z_1+\delta z_2, \hskip 3mm
\mbox{if
$g=\begin{bmatrix}\alpha&\beta\\
\gamma&\delta\end{bmatrix}\in G$.}
\end{equation}
In this case $\dim G=3$, $|W|=2$, $r=1$,
$m_1=1$, $E=\mathbb R$, $X(T)=\mathbb Z$. The
set $R_+$ consists of a single root $\alpha=2$
and $\alpha\!^{\vee}$ is the standard coordinate
function on $\mathbb R$, i.e.,
$\alpha\!^{\vee}(a)=a$ for every $a\in \mathbb
R$. Take the sequence of monomials
 $z_1^h, z_1^{h-1}z_2,\ldots,
 z_1z_2^{h-1}, z_2^h$ as
 a basis $e_1,\ldots, e_{h+1}$ of $V$. It follows from
\eqref{degree} that $e_{i+1}$ is a weight vector of
the diagonal torus $T=\{{\rm diag}(t, t^{-1})\mid t\in
k\setminus\{0\}\}$ with weight $t\mapsto t^{h-2i}$.
Therefore the weight system of the $T$-module $V$ is
the arithmetic progression  $\{h, h-2,\ldots, -h+2,
-h\}$. Hence $\mathcal P_{V_h}=[-h, h]$. The kernel of
the representation $\rho=\rho_h$ given by formula
\eqref{matr} is trivial if  $h$ odd and has order $2$
if $h$ is even. Therefore we deduce from
\eqref{degree} that
\begin{equation}\label{degsl2}
{\rm deg}\,\rho_h({\rm SL}_2)= \frac{3!}{2|\ker
\rho_h|}\int_{-h}^{h}x^2dx=
\begin{cases}
2h^3\quad &\mbox{if $h$ is odd},\\
h^3\quad& \mbox{if $h$ is even.}
\end{cases}
\end{equation}
\end{example}


\end{nothing*}

\begin{nothing*}\label{reduction} {\it
The reduction to conic case}

\vskip 1mm

Let $L$ be a finite dimensional vector space
over $k$. Let $H$ be an algebraic group
(algebraically) acting on a projective space
$\mathbb P(L)$ of one-dimensional linear
subspaces of $L$. Keeping $H$-orbits in $\mathbb
P(L)$, we may replace the group $H$ by its
quotient group by the kernel of action and
assume that $H$ is a subgroup of ${\rm
Aut}(\mathbb P(L))$. Let $\widetilde H$ be the
inverse image of  $H$ with respect to the
natural homomorphism  ${\rm GL}(L)\to {\rm
Aut}(\mathbb P(L))$ (note that $H$ is reductive
if and only if $\widetilde H$ shares this
property). Let $\pi\colon L\setminus \{0\}\to
\mathbb P(L)$ be the natural projection. We call
a subset in $L$ {\it conic} if it is stable with
respect to scalar multiplication by every
nonzero element $k$.
\begin{lemma}\label{pro} Let
$U$ be a nonempty open  $H$-stable subset of
$\mathbb P(L)$ and let $p, q\in U$ be two its
points. Take any points $\widetilde p\in
\pi^{-1}(p)$ and $\widetilde q\in \pi^{-1}(q)$.
Then the following properties are equivalent:
\begin{enumerate}
\item[(i)] the orbit $H\cdot p$ lies in the
closure of the orbit $H\cdot q$ in $\mathbb
P(L)$; \item[(ii)] the orbit $H\cdot p$ lies in
the closure of the orbit  $H\cdot q$ in
$U$;\item[(iii)] the orbit $\widetilde H\cdot
\widetilde p$ lies in the closure of the orbit
$\widetilde H \cdot \widetilde q$ in $L$.
\end{enumerate}
The orbits $\widetilde H\cdot \widetilde p$ and
$\widetilde H\cdot \widetilde q$ are conic.
\end{lemma}

\begin{proof} As $\widetilde H$ contains all
scalar multiplications of the space $L$ by
nonzero scalars, $\widetilde H\cdot\widetilde
p=\pi^{-1}(H\cdot p)$ and $\widetilde
H\cdot\widetilde q=\pi^{-1}(H\cdot q)$. The
reader will easily check that the statement
follows from this and the definitions.
\end{proof}

We shall need the following application of Lemma
\ref{pro}. Let $L$ be the coordi\-na\-te space
 $k^{n+1}$ and let  $H$ be the group  $G$
from Subsectionиз \ref{begin} that acts on
$\mathbb P(L)$ according to the rule (see
\eqref{matr})
\begin{equation*}\label{action}
g\cdot (\alpha_0: \alpha_1:\ldots:\alpha_n):=
\biggl(\alpha_0:\sum_{i=1}^n\rho_{1,i}(g)
\alpha_i:\ldots:
\sum_{i=1}^n\rho_{n,i}(g)\alpha_i\biggr).
\end{equation*}
The standard principal open subset
$\{(\alpha_0:\alpha_1:\ldots:\alpha_n)\mid
\alpha_0\neq 0\}$ of $\mathbb P(L)$ is $G$-stable and
is equivariantly isomorphic to the $G$-module $V$.
Hence, by Lemma \ref{pro}, answering question $(*)$ is
equivalent to answering
whether or not  the orbit $\widetilde G\cdot
\widetilde a$ lies in the closure of the orbit
$\widetilde G\cdot \widetilde b$. This means
that replacing the group $G$ by $\widetilde G$,
the space $V$ by $L$, and the points $a$ and $b$
by, respectively, $\widetilde a$ and $\widetilde
b$, we reduce solving Problem  $(*)$ to the case
where both orbits  $G\cdot a$ and $G\cdot b$ are
nonzero and conic. Given this,
\begin{gather}\label{cone}
\begin{gathered}
\mbox{\rm searching for an answer to Problem
$(*)$, we may assume that
}\\[-3pt] \mbox{\rm\hskip -37.5mm
$G\cdot a$ and $G\cdot b$ are nonzero conic
orbits}.
\end{gathered}
\end{gather}

Note also that by Lemma \ref{pro} the problem
analogous to Problem  $(\ast)$, but for an
action on a projective space is reduced to
Problem  $(\ast)$ for an action on a linear
space.

\end{nothing*}

\begin{nothing*}
{\it The input of the algorithm}

\vskip 1mm

We assume that the following data are known (cf.\,\cite{Po81}):
\begin{list}{
} {\labelsep=2mm\leftmargin=8.5mm
\itemindent=-1mm\labelwidth=1mm\parsep=1mm}
\item[---] The degree of the variety $\rho(G)$,
\begin{equation}\label{d}
d:={\rm deg}\,\rho(G).
\end{equation}
 \item[---] The functions $$\iota^*(\rho_{p, q})\!\in\!
 k[\mathbb A^{r, s}]\!=\!k[x_1,\ldots, x_{r+s},
x_1^{-1},\ldots,x_{r}^{-1}], \qquad 1\leqslant
p, q\leqslant n.$$
\end{list}
\end{nothing*}

\begin{example}\label{ex} Consider
the same situation as in Example \ref{sl2}. Number
\eqref{d} is given by formula \eqref{degsl2}. It
follows from \eqref{sl} that the functions
$\rho_{p,q}$ in \eqref{matr} are defined by the
equality
\begin{equation}\label{rh}
(\alpha z_1+\gamma z_2)^{h-j}(\beta z_1+\delta
z_2)^{j}\!=\!\sum_{i=0}^{h}\rho_{i+1,
j+1}(g)z^{h-i}_1z^{i}_2,\hskip 3mm \mbox{где
$g\!=\!\begin{bmatrix}\alpha&\beta\\
\gamma&\delta\end{bmatrix}\!\in\! G$}.
\end{equation}

Take $\iota$ to be the morphism
\begin{gather}
\iota\colon \mathbb A^{1, 2}\hookrightarrow {\rm SL}_2(k),\notag\\
(\varepsilon_1,
\varepsilon_2,\varepsilon_3)\mapsto
\begin{bmatrix}1&\varepsilon_2\\
0&1\end{bmatrix}
\begin{bmatrix}\varepsilon_1&0\\
0&\varepsilon_1^{-1}\end{bmatrix}
\begin{bmatrix}1&0\\
\varepsilon_3&1\end{bmatrix}=
\begin{bmatrix}
\varepsilon_1^{-1}\varepsilon_2\varepsilon_3+\varepsilon_1
&\varepsilon_1^{-1}\varepsilon_2\\
\varepsilon_1^{-1}\varepsilon_3&\varepsilon_1^{-1}\end{bmatrix}.\label{iosl}
\end{gather}
Then it follows from \eqref{coord}, \eqref{rh},
\eqref{iosl} that the function $\iota^*(\rho_{i+1,
j+1})$ is equal to the coefficient  $z^{h-i}_1z^{i}_2$
in the decomposition of the binary form
\begin{equation*}
\bigl((x_1+x^{-1}_{1}x_2x_3)z_1+(x^{-1}_{1}x_3)
z_2\bigr)^{h-j}\bigl((x^{-1}_{1}x_2) z_1+(x^{-1}_{1})
z_2\bigr)^{j}
\end{equation*}
in the variables $z_1, z_2$ with the coefficients in
the field $k(x_1, x_2, x_3)$ as sum of monomials in
$z_1, z_2$. For instance, if $h=2$, then
$\iota^*(\rho_{2, 2})= 1+2x^{-2}_{1}x_2x_3$.
\end{example}

\begin{nothing*}\label{algor} {\it The algorithm}

\vskip 1mm

Now we turn to the formulation and proof of the
main result. We utilize the notation and
conventions introduced above and exclude the
trivial case $\overline{G \cdot  b}=V$, i.e.,
assume that
\begin{equation}\label{dim}
\dim\,G\cdot b<\dim\, V
\end{equation}
(as the number $\dim\,G\cdot b$ is equal to the
rank of the system of vectors $\{d\rho(Y_i)\cdot
b\}_{i\in I}$ where $\{Y_i\}_{i\in I}$ is a
basis of the vector space ${\rm Lie}(G)$,
condition \eqref{dim} can be verified
constructively if  the operators
 $d\rho(Y_i)$ are known).

The following sequence of steps together with
Theorem \ref{thm} provide a const\-ruc\-ti\-ve
method to answer Problem $(\ast)$:

\begin{list}{
}
{\labelsep=2mm\leftmargin=10mm
\itemindent=-1mm\labelwidth=4mm\parsep=1mm}
\item[(1)]
Find the coordinates of the vectors $a$ and $b$
in the basis $e_1,\ldots, e_n$:
\begin{equation*}a=\alpha_1 e_1+\ldots+\alpha_n
e_n,\hskip 3mm b=\beta_1e_1+\ldots+\beta_n
e_n,\end{equation*} and, changing the basis
$e_1,\ldots, e_n$  if necessary, achieve that
the following condition holds:
\begin{equation}\label{nonzero}
\beta_1\cdots\beta_n\neq 0.
\end{equation}
\item[(2)]
Consider $n$ ``generic'' polynomials
$F_1,\ldots, F_n$ of degree $2d-2$ (where $d$ is
defined by formula \eqref{d}) in the variables
$y_1,\ldots, y_n$,
\begin{align}\label{f}
F_p:=\sum_{\substack{q_1,\ldots, q_n\in\mathbb N\\
q_1+\cdots+q_n\leqslant 2d-2}} c_{p, q_1,\ldots,
q_n}y^{q_1}_1\cdots y^{q_n}_n, \hskip 3mm
p=1,\ldots, n,
\end{align}
i.e., such that, apart from  $y_1,\ldots, y_n$,
all the coefficients $c_{p, q_1,\ldots, q_n}$
are the indeterminates over $k$ as well, and put
\begin{equation}\label{F}
H(y_1,\ldots,
y_n):=(y_1-\alpha_1)F_1+\ldots+(y_n-\alpha_n)F_n-1.
\end{equation}
\item[(3)]
Replacing in $H(y_1,\ldots, y_n)$ every variable
$y_i$ by $\sum_{j=1}^n\beta_j\iota^*(\rho_{i,
j})$, obtain a linear combination of monomials
of form  \eqref{monom} with the coefficients in
the ring $k[\ldots, c_{p, q_1,\ldots,
q_n},\ldots]$ of polynomials in variables $c_{p,
q_1,\ldots, q_n}$ over the field $k$:

\begin{align}\label{subst}
H\biggl(\sum_{j=1}^n\beta_j\iota^*(\rho_{1,
j}),\ldots, \sum_{j=1}^n&\beta_j\iota^*(\rho_{n,
j})\biggr) \notag\\[- 6pt]
&=\sum_{(i_1,\ldots, i_{r+s})\in M}
\ell_{i_1,\ldots, i_{r+s}}x^{i_1}_{1}\cdots
x^{i_{r+s}}_{r+s},
\end{align}

\noindent where $\ell_{i_1,\ldots, i_{r+s}}\in
k[\ldots, c_{p, q_1,\ldots, q_n},\ldots]$ and
$M$ is a finite subset in  $\mathbb Z^r\times
\mathbb N^s$. By \eqref{f} and \eqref{F}, every
coefficient  $\ell_{i_1,\ldots, i_{r+s}}$ in
\eqref{subst} is a  {\it linear} function in the
variables  $c_{p, q_1,\ldots, q_n}$ with the
coefficients in the field $k$. \item[(4)]
Consider the following finite system of {\it
linear} equations in the variables\break $c_{p,
q_1,\ldots, q_n}$ with the coefficients in the
field $k$:
\begin{equation}\label{ls}
\ell_{i_1,\ldots, i_{r+s}}=0, \hskip
3mm\mbox{where $(i_1,\ldots, i_{r+s})\in M$.}
\end{equation}
\end{list}
\begin{theorem}\label{thm} Let
$G\cdot b$ be a nonzero conic orbit {\rm(}see
\eqref{cone}{\rm)}. The following properties are
equivalent:
\begin{enumerate}
\item[(i)]  the closure of the orbit $G\cdot b$
in $V$ contains the orbit  $G\cdot a$;
\item[(ii)] system of linear equations
\eqref{ls} is inconsistent.
\end{enumerate}
\end{theorem}
\begin{proof} We split it into several steps.

1.\hskip 2mm Let $z_1,\ldots, z_n$ be the basis
of $V^*$ dual to $e_1,\ldots, e_n$. Let $t_i$ be
the restriction to $\overline{G\cdot b}$ of the
function  $z_i$. As the set $\overline{G\cdot
b}$ is closed in $V$ and $k[V]=k[z_1,\ldots,
z_n]$, we have
\begin{equation}\label{tt}
k[\overline{G \cdot b}]=k[t_1,\ldots, t_n].
\end{equation}

Note that the function $t_i$ is not a constant.
Indeed, as the orbit $G\cdot b$ is conic, $0\in
\overline{G \cdot b}$. The definition of $t_i$
implies that $t_i(0)=0$ and $t_i(b)=\beta_i$.
But $\beta_i\neq 0$ because of \eqref{nonzero}.

Consider the orbit morphism
\begin{equation}\label{morphism}
\varphi\colon G\to \overline{G \cdot  b},\quad
\varphi(g)=g \cdot  b.
\end{equation}
As the morphism $\iota$ is dominant, the image
of the morphism
\begin{equation}\label{psi}
\psi:=\varphi\circ\iota\colon \mathbb A^{r,
s}\to \overline{G \cdot b}
\end{equation}
is dense in $\overline{G \cdot  b}$; whence the
corresponding comorphism is an embedding of the
algebra of regular functions:
\begin{equation}\label{psi*}
\psi^*\colon k[\overline{G \cdot
b}]\hookrightarrow k[\mathbb A^{r, s}].
\end{equation}
It follows from \eqref{tt} that
\begin{equation}
\psi^*(k[\overline{G \cdot
b}])=k[
\psi^*(t_1),\ldots,
\psi^*(t_n)],
\end{equation}
and \eqref{morphism}, \eqref{psi},
\eqref{actionn}, and the definition of $t_i$
imply that
\begin{equation}\label{ip}
\psi^*(t_i)=\sum_{j=1}^n\beta_j\iota^*(\rho_{i,j}).
\end{equation}

There is only one point of $V$ where all the
functions $z_1-\alpha_1,\ldots, z_n-\alpha_n$
vanish, the point $a$. Taking into account that
the set $\overline{G \cdot b}$ is $G$-stable, we
deduce from this that the following properties
are equivalent:
\begin{equation}\label{equiv}
\left. \begin{array}{rl} {\rm({\rm c}_1)}&\hskip
-2mm\mbox{the orbit $G \cdot  a$ does not lie in
the closure
of the orbit  $G \cdot  b$};\\
{\rm({\rm c}_2)}&\hskip -2mm\mbox{the point $a$
does not lie in the closure of the orbit
$G \cdot  b$};\\
{\rm ({\rm c}_3)}&\hskip -2mm\mbox{there are no
points of $\overline{G \cdot b}$ where all the
functions}\\
{}&\hskip -2mm\mbox{ $t_1\!-\!\alpha_1,\ldots,
t_n\!-\!\alpha_n$ vanish}.
\end{array}
\right
\}
\end{equation}

2.\hskip 2mm Now we shall use the effective form
of Hilbert's Nulltellensatz obtained in
\cite{J05}. In order to formulate this result we
shall introduce some notation and definitions.

First, for every positive integers $d_1\geqslant
\cdots\geqslant d_m$ and  $q$, set

\begin{equation*}
N(d_1,\ldots, d_m; q)=\begin{cases}\prod_{i=1}^m
d_i\hskip 2mm
&\mbox{if $q\geqslant m\geqslant 1$},\\
\hskip -.5mm\bigl(\prod_{i=1}^{q-1}d_i\bigr)d_m
\hskip 2mm
&\mbox{if $m>q>1$},\\
d_m \hskip 2mm &\mbox{if $q=1$},
\end{cases}
\end{equation*}
and also
\begin{equation*}
N'(d_1,\ldots, d_m; q)=N(d_1,\ldots, d_m; q)
\hskip 1mm \mbox{if $q>1$},\hskip 3mm \mbox{and
}\hskip 1mm N'(d_1,\ldots, d_m; 1)=d_1.
\end{equation*}

Further, if
a
nonzero regular function
 $h$
on an irreducible closed subset
 $X$
of an affine space  $\mathbb A^l$ is given,  the
minimum of degrees of polynomial functions on
$\mathbb A^l$ whose restriction to $X$ is $h$
will be called the {\it degree} of $h$ and
denoted by ${\rm deg}\,h$. It is easily seen
that, for every nonzero regular functions  $f$
and $h$ on $X$, the inequality ${\rm
deg}\,fh\leqslant {\rm deg}\,f+{\rm deg}\,h$
holds and, in general,
it may be strict. However, if the set $X$ is
conic, then necessarily
\begin{equation}\label{=}
{\rm deg}\,fh={\rm deg}\,f+{\rm deg}\,h.
\end{equation}

\begin{theorem}[\rm Z. Jelonek \cite{J05}]
\label{Jel} Let $X$ be an irreducible closed
subset of  $\mathbb A^l$ of positive dimension.
Let $h_1,\ldots, h_m$ be the nonconstant regular
functions on $X$ such that
\begin{equation}\label{ineqv}
{\rm deg}\,h_1\geqslant\cdots\geqslant {\rm
deg}\,h_m.
\end{equation}
Then the following conditions are equivalent:
\begin{enumerate}
\item[\rm (a)] there are no points of~$\;X$
where all the functions $h_1,\ldots, h_m$
vanish; \item[\rm(b)] there are regular
functions $f_1,\ldots, f_m$ on $X$ such that
$$1=\sum_{i=1}^{m} f_ih_i$$ and, for every $i$,
the following inequality holds:
\begin{equation*}
{\rm deg}\,f_i h_i\leqslant\begin{cases} {\rm
deg}\,X\!\cdot\!N'({\rm deg}\,h_1,\ldots,{\rm
deg}\,h_m; \dim\,X),\hskip 2mm &\mbox{if
$\;m\leqslant
\dim\,X$},\\
2{\rm deg}\,X\!\cdot\!N'({\rm
deg}\,h_1,\ldots,{\rm deg}\,h_m;
\dim\,X)-1,\hskip 2mm &\mbox{if $\;m> \dim\,X$.}
\end{cases}
\end{equation*}
\end{enumerate}
\end{theorem}

Now we take as $\mathbb A^l$ and $X$
respectively $V$ and $\overline{G \cdot b}$. As
the nonconstant function $t_i$  is the
restriction to $\overline{G \cdot b}$ of a
linear function on $V$, we have
\begin{equation}\label{h}
{\rm deg}\,(t_1-\alpha_1)=\ldots={\rm
deg}\,(t_n-\alpha_n)=1.
\end{equation}
Further, in Theorem \ref{Jel} put $m=n$ and
$h_i=t_i-\alpha_i$, $i=1,\ldots, n$ (by
\eqref{h} condition \eqref{ineqv} is fulfilled).
Then it follows from \eqref{equiv}, Theorem
\ref{Jel}, and \eqref{dim}, \eqref{h} that every
property  $({\rm c}_1)$, $({\rm c}_2)$, $({\rm
c}_3)$ in \eqref{equiv} is equivalent to the
property
\begin{enumerate}
\item[$({\rm c}_4)$] there are
functions $f_1,\ldots, f_n$ regular on
$\overline{G \cdot  b}$ such that\\
$\sum_{i=1}^{m} (t_i-\alpha_i)f_i-1=0$ and, for
every $i$, the following inequality holds:
\begin{equation}\label{degg}
{\rm deg}\,(t_i-\alpha_i)f_i\leqslant 2\,{\rm
deg}\, \overline{G \cdot  b}-1.
\end{equation}
\end{enumerate}

As $\overline{G \cdot  b}$ is a conic
(irreducible) subvariety of $V$, it follows from
\eqref{=} and \eqref{h} that ${\rm
deg}\,(t_i-\alpha_i)f_i=1+{\rm deg}\,f_i$.
Therefore inequality \eqref{degg} is equivalent
to the inequality
\begin{equation}\label{deggg}
{\rm deg}\,f_i\leqslant 2\,{\rm deg}\,
\overline{G \cdot  b}-2.
\end{equation}

3.\hskip 2mm The degree of the variety
$\overline{G \cdot b}$ can be upper bounded.
Namely, the orbit ${G \cdot b}$ is the image of
the variety  $\rho (G)\subset {\rm End}(V)$
under the linear map  ${\rm End}(V)\!\to V$,
$g\mapsto g\cdot b$. But it is easy to prove
(see, e.g., \cite[Prop. 4.7.10]{DK}) that
degree does not increase under affine maps: if
$Y$ is a locally closed subset of $\mathbb A^l$
and $\varphi\colon \mathbb A^l\to\mathbb A^m$ is
an affine map, then ${\rm deg}\,Y\geqslant {\rm
deg}\,\overline{\varphi(Y)}$. Therefore ${\rm
deg}\, \overline{G \cdot  b}$ is not bigger than
the degree of the subvariety  $\rho(G)$ in ${\rm
End}(V)$, i.e., the number $d$.
 Hence, by virtue of
\eqref{deggg}, for every  $i$, the following
inequality holds
\begin{equation}\label{degggg}
{\rm deg}\,f_i\leqslant 2d-2.
\end{equation}

4.\hskip 2mm As the comorphism  $\psi^*$ (see
\eqref{psi*}) is an embedding, we obtain the
equiva\-lence
\begin{equation}\label{eqv}
\sum_{i=1}^{m} (t_i-\alpha_i)f_i-1=0\hskip 1.5mm
\iff\hskip 1.5mm  \sum_{i=1}^{m}
(\psi^*(t_i)-\alpha_i)\psi^*(f_i)-1=0.
\end{equation}

By virtue of \eqref{ip} and \eqref{eqv}, it
follows from inequality \eqref{degggg} and the
definitions of functions $t_i$, numbers ${\rm
deg}\,f_i$, the ``generic'' polynomials $F_p$
(see \eqref{f}), and the polynomial $H$ (see
\eqref{F}) that property $({\rm c}_4)$ is
equivalent to the following property:
\begin{enumerate}
\item[$({\rm c}_5)$] for every coefficient
$c_{p, q_1,\ldots, q_n}$ of every ``generic''
polynomial $F_p$, there is a constant $\nu_{p,
q_1,\ldots, q_n}\in k$ such that after
substitution of $\nu_{p, q_1,\ldots, q_n}$ in
place of $c_{p, q_1,\ldots, q_n}$ for every $p,
q_1,\ldots, q_n$, the right-hand side of formula
\eqref{ls} becomes zero of the field of rational
functions in $x_1,\ldots, x_{r+s}$ with the
coefficients in $k$:
\begin{equation}\label{lll}
\sum_{(i_1,\ldots, i_{r+s})\in M}
\ell_{i_1,\ldots, i_{r+s}}(\ldots,\nu_{p,
q_1,\ldots, q_n},\ldots)\, x^{i_1}_{1}\cdots
x^{i_{r+s}}_{r+s}=0.
\end{equation}
\end{enumerate}

It remains to notice that as monomials
$x^{i_1}_{1}\cdots x^{i_{r+s}}_{r+s}$, where
$(i_1,\ldots, i_{r+s})\in M$, are linearly
independent over $k$, equality \eqref{lll} is
equivalent to vanishing of all the coefficients
of the left-hand side,
$$
\ell_{i_1,\ldots, i_{r+s}}(\ldots,\nu_{p,
q_1,\ldots, q_n},\ldots)=0,
$$
i.e., to that  $c_{p, q_1,\ldots, q_n}=\nu_{p,
q_1,\ldots, q_n}$ is a solution of system of
linear equations \eqref{ls} in variables  $c_{p,
q_1,\ldots, q_n}$. This completes the proof of
the theorem.
\end{proof}

\begin{remark} The proof shows that
the claim of Theorem  \ref{thm} remains true if
the constant  $d$ in the definition of the
``generic'' polynomials $F_1,\ldots, F_n$ is
replaced by ${\rm deg}\,\overline{G\cdot b}={\rm
deg}\,G\cdot b$. If, from some reasons, the
number ${\rm deg}\,G\cdot b$ is known,  this
permits to decrease the number of variables and
equations in the system of linear equations
уравнений \eqref{ls}. In some cases the degrees
of orbits indeed have been computed.
\end{remark}

\begin{example} Consider the same
situation as in Examples \ref{sl2} and \ref{ex}. Take
a nonzero binary form $v\in V_h$ and decompose it as a
product $v=v_{1}^{n_1}\cdots v_{p}^{n_p}$, where
$v_1,\ldots, v_p$ are pairwize nonproportional forms
from  $V_1$. Assu\-me that  $p\geqslant 3$ and
$h/n_i\geqslant 2$ for every $i$. Then the
$G$-stabilizer $G_v$ of the form  $v$ is finite
 \cite{Po74} and  $|G_v|{\rm deg}\,G\cdot v
=-2(p-1)h^3-4\sum_{i=1}^p(h-n_i)^3+3h^2
\sum_{i=1}^p(h-n_i)+ 3h\sum_{i=1}^p(h-n_i)(h-2n_i)$
(see the proof in
 \cite[Sect.\,8]{MJ92}). In particular,
 if all the roots of the form
  $v$ are simple, i.e.,   $p=h$,
$n_1=\ldots=n_h=1$, then
\begin{equation}\label{dG}
|G_v|{\rm deg}(G\cdot v)=2h(h-1)(h-2).
\end{equation}
Formula \eqref{dG} can also be deduced from a
calculation made in  1897 by Enriques and Fano; this
has been done in 1983 by Mukai and Umemura
 (with a gap fixed in
 \cite[Sect.\,8, Remark]{MJ92} where one can
 find
 the relevant references).
\end{example}
\end{nothing*}

\section{\bf Defining the set \boldmath$
\overline{G\cdot L}$ by equations}

 \begin{nothing*} Let $L$ be a linear subvariety
 of $V$. Then there is a morphism
\begin{equation*}\label{t}
\tau\colon \mathbb A^l\to V,
\end{equation*}
whose image is dense in  $L$: for instance,
 one can take $\tau$ to be an affine embedding of
$\mathbb A^l$ into $V$ whose image is $L$. We
fix such a morphism  $\tau$. Besides, like above
we assume that a dominant morphism \eqref{iota}
is fixed.

We maintain the notation from Subsection
\ref{o}. Like in the proof of Theo\-rem
\ref{thm}, we denote by  $z_1,\ldots, z_n$ the
basis of $V^*$ dual to $e_1,\ldots, e_n$.
Besides, we denote by $y_1,\ldots, y_l$ the
standard coordinate functions on  $\mathbb A^l$:
\begin{equation*}
y_i(a)=\delta_i \quad \mbox{for
$a=(\delta_1,\ldots, \delta_l)\in\mathbb A^l$}.
\end{equation*}
Then
\begin{equation}\label{tau}
\tau(v)=\sum_{i=1}^{n}\tau^*(z_i)(v)e_i\quad
\mbox{for every $v\in\mathbb A^l$.}
\end{equation}
\end{nothing*}

\begin{nothing*}
The functions  $x_1,\ldots, x_{r+s}$,
$y_1,\ldots, y_l$ can be naturally extended to
the functi\-ons on $\mathbb A^{r, s}\times
\mathbb A^l$; we denote these extensions by the
same letters. Consider the morphism
\begin{equation}\label{muuu}
\mathbb \mu\colon A^{r, s}\times \mathbb A^l\to
V,\quad \mu(u, v)=\iota(u)\cdot\tau(v)
\end{equation}
Then \eqref{actionn} and \eqref{tau} imply that
\begin{equation}\label{m*}
f_p:=\mu^*(z_p)=\sum_{q=1}^{n}\iota^*(\rho_{pq})
\tau^*(z_q),\qquad 1\leqslant p, q\leqslant n.
\end{equation}
 \end{nothing*}

We identify $\mathbb A^{r, s}\times \mathbb A^l$
with the open subset of  $\mathbb A^{r+s+l}$ by
means of the embedding
$$\mathbb A^{r, s}\times \mathbb A^l\hookrightarrow
\mathbb A^{r+s+l},\quad
((\varepsilon_1\ldots,\varepsilon_{r+s}),
(\delta_1,\ldots,\delta_l))\mapsto
(\varepsilon_1\ldots,\varepsilon_{r+s},
\delta_1,\ldots,\delta_l);
$$
then 
$x_1,\ldots, x_{r+s}$, $y_1,\ldots, y_l$ become the
standard coordinate functions on $\mathbb A^{r+s+l}$.
Besides, we identify  $V$
with $\mathbb A^n$ by means of the isomorphism
$$
V\to \mathbb A^n,\qquad \sum_{i=1}^n
\gamma_ie_i\mapsto (\gamma_1,\ldots,\gamma_n).
$$
Then morphism \eqref{muuu} becomes the rational
map $\varrho$ of the affine space $\mathbb
A^{r+s+l}$ to the affine space $\mathbb A^n$:
$$\varrho\colon \mathbb A^{r+s+l}
\dashrightarrow \mathbb A^n,\quad a\mapsto
(f_1(a),\ldots,f_n(a)).$$  As $\overline{\iota
(\mathbb A^{r,s})\cdot L}=\overline{G\cdot L}$,
we have the equality
\begin{equation}\label{GL}
\overline{\varrho(\mathbb
A^{r+s+l})}=\overline{G\cdot L}.
\end{equation}

\noindent This makes it possible to apply eliminaton
theory to finding the equations that cut out
$\overline{G\cdot L}$ in $V$. An algorithmic solution
to this problem is obtained by means of Gr\"obner
bases as follows.

\begin{nothing*} {\it The input of the algorithm}

\vskip 1mm

We assume that the following data are known:
\begin{list}{
}
{\labelsep=2mm\leftmargin=8.5mm
\itemindent=-1mm\labelwidth=1mm\parsep=1mm}
 \item[---] The functions
 \begin{equation}\label{irpq}
 \iota^*(\rho_{p, q})\!\in\!
 k[x_1,\ldots, x_{r+s},
x_1^{-1},\ldots,x_{r}^{-1}]\subset k(\mathbb
A^{r+s+l}), \qquad 1\leqslant p, q\leqslant n.
\end{equation}
\item[---] The functions
\begin{equation}\label{tz}
\tau^*(z_i)\in k[y_1,\ldots, y_l]\subset
k[\mathbb A^{r+s+l}],\qquad 1\leqslant
i\leqslant n.
\end{equation}
\end{list}

\begin{example} Fix a point
$v\in L$ and a sequence  $f_1,\ldots, f_m$ of
linear indepen\-dent vectors defining a
parametric presentation
$L=\{v+\sum_{i=1}^m\lambda_if_i\mid
\lambda_d,\ldots,\lambda_m\break \in k\}$. Take
$\tau$ to be the embedding  $ \tau\colon \mathbb
A^m\hookrightarrow \mathbb A^n$,
$\tau(\lambda_1,\ldots,\lambda_m)=
v+\sum_{i=m}^l\lambda_if_i. $ Let
$v=\sum_{j=1}^n\gamma_je_j$ and $f_i=
\sum_{j=1}^n\nu_{ji}e_j$. Then $$
\tau^*(z_i)=\sum_{i=1}^n
\nu_{ij}y_j+\gamma_i,\quad 1\leqslant i\leqslant
n.
$$  See Example \ref{ex} regarding the functions
$\iota^*(\rho_{p, q})$.
\end{example}
\end{nothing*}
\begin{nothing*} {\it The algorithm}

\vskip 1mm

The following sequence of steps together with
Theorem \ref{thmm} provide a const\-ruc\-ti\-ve
method to obtain the equations defining
$\overline{G\cdot L}$ in $V$:

\begin{list}{}
{\labelsep=2mm\leftmargin=10mm
\itemindent=-1mm\labelwidth=4mm\parsep=1mm}
\item[(1)] Compute the rational functions $f_p$
using formula \eqref{m*} and write down each of
them as a fraction of polynomials:
$$
f_p=\frac{g_p}{h_p},\quad \mbox{где\hskip 3mm
$g_p\in k[x_1,\ldots, x_{r+s}, y_1,\ldots,
y_t]$,\; $h_p\in k[x_1,\ldots, x_r]$}
$$
(see \eqref{irpq} and \eqref{tz}). \item[(2)]
Consider the polynomial ring  $k[t,
x_1,\ldots,x_{r+s}, y_1,\ldots, y_t, z_1,\ldots,
z_n]$, whe\-re $t$ is a new variable, and find
for its ideal generated by the polynomials
\begin{equation*}
h_1z_1-g_1,\ldots, h_nz_n-g_n, 1-h_1\cdots h_nt.
\end{equation*}
a Gr\"obner basis with respect to an order of
monomials such that every variable   $t$, $x_i$
и $y_j$ is bigger than every variable  $z_p$.
\end{list}
\begin{theorem}\label{thmm}
Let $q_1,\ldots, q_m$ be all the elements of
this Gr\"obner basis that lie in  $k[z_1,\ldots,
z_n]$. Then
\begin{equation*}
\overline{G\cdot L}=\{v\in \mathbb A^n\mid
q_1(v)=\ldots=q_m(v) =0\}.
\end{equation*}
\end{theorem}
\begin{proof} We have $\overline{\varrho
(\mathbb A^{r+s+l})}\!=\!\{v\!\in\!\mathbb
A^n\mid q_1(v)\!=\!\ldots\!=\!q_m(v)
\!=\!0\}$---this is a general fact about the
closure of image of every rational map of one
affine space to another, see,
e.g.,\;\cite[Chap.\,3, \S\,3, Theorem 2]{CLO98}.
Now the claim that we wish to prove  follows
from equality~\eqref{GL}.
\end{proof}
\begin{remark}
Although the elements $q_1,\ldots, q_m$, interesting
for us, constitute a part of the Gr\"obner basis,  for
finding them by means of the described algorithm, we
have to find the whole of this basis.
\end{remark}
\end{nothing*}



\end{document}